\documentclass[12pt,a4paper]{amsart}
\usepackage{amssymb}
\usepackage{tikz}

\newtheorem{thm}{Theorem}
\newtheorem{dfn}{Definition}
\theoremstyle{definition}
\newtheorem{exm}{Example}

\newcommand{\mC}{{\mathbb{C}}}

\title[The cocycle condition for multi-pullbacks of algebras]{THE COCYCLE CONDITION\\ FOR MULTI-PULLBACKS OF ALGEBRAS}
\vspace*{-25mm}

\author{Piotr~M.~Hajac}
\address{Instytut Matematyczny,
Polska Akademia Nauk,
ul.~\'Sniadeckich 8, Warszawa, 00-956 Poland\\
Katedra Metod Matematycznych Fizyki,
Uniwersytet Warszawski,
ul. Ho\.za 74, Warszawa, 00-682 Poland}
\email{http://www.impan.pl/\~{}pmh}

\author{Bartosz~Zieli\'nski}
\address{Department of Theoretical Physics and Computer Science, University of \L{}\'od\'z,
Pomorska 149/153 90-236 \L{}\'od\'z, Poland}
\email{bzielinski@uni.lodz.pl}

\begin{document}
\baselineskip=14.5pt

\begin{abstract}
Take finitely many topological spaces and for each pair of these spaces
choose a pair of corresponding closed subspaces that are identified
by a homeomorpism. We note that this gluing procedure does not guarantee
that the building pieces, or the gluings of some pieces, are embedded in the  space obtained by putting together all given ingredients.
Dually, 
 we show that a certain sufficient condition, called
the cocycle condition, is also necessary  
to guarantee sheaf-like properties of surjective multi-pullbacks of algebras with distributive lattices of 
ideals.
\end{abstract}

\maketitle

\noindent
When constructing a topological space as the gluing of pieces, it is desirable that the parts are embedded into the described space.
The gluing of  three intervals $I_1\cong I_2\cong I_3\cong [-1,1]$ 
into the space $T_*$ described by Fig.~\ref{fig0}
fails this property as the endpoints of $I_2$ and $I_3$ are glued  into a single point:
\begin{figure}[!htp]
\begin{center}
\begin{tabular}{ccc}
\begin{tikzpicture}[scale=0.8]
\draw[dashed,very thin,red] (0,0) circle (5mm);
\draw[color=white,very thick] (50:5mm) arc (50:130:5mm);
%\draw[dashed,very thin,red] (50:20mm) arc (50:130:20mm);
\draw[dashed,very thin,red] (50:20mm)--(130:5mm);
\draw[dashed,very thin,red] (50:5mm)--(130:20mm);
\draw (0,-5mm) -- (0,-20mm);
\draw (50:5mm) -- (50:20mm);
\draw (130:5mm) -- (130:20mm);
\fill[red] (0,-5mm) circle (0.5mm);
\fill[red] (50:5mm) circle (0.5mm);
\fill[red] (130:5mm) circle (0.5mm);
\fill[red] (50:20mm) circle (0.5mm);
\fill[red] (130:20mm) circle (0.5mm);
\draw (0,-5mm) node[anchor=north west]{$\scriptstyle1$};
\draw (0,-20mm) node[anchor=west]{$\scriptstyle-1$};
\draw (0,-13mm) node[anchor=east]{$\scriptstyle I_1$};
\draw (130:5mm) node[anchor=east]{$\scriptstyle1$};
\draw (130:20mm) node[anchor=east]{$\scriptstyle-1$};
\draw (130:13mm) node[anchor=north east]{$\scriptstyle I_2$};
\draw (50:5mm) node[anchor=west]{$\scriptstyle1$};
\draw (50:20mm) node[anchor=west]{$\scriptstyle-1$};
\draw (50:13mm) node[anchor=north west]{$\scriptstyle I_3$};
\end{tikzpicture}&& 
\begin{tikzpicture}[scale=0.8]
\draw (0,0) -- (0,-14mm);
\draw (0,-14mm) node[anchor=north]{$\scriptstyle-1$};
\draw (0,0) -- (30:14mm) arc (30:60:14mm)--(0,0);
\draw (0,0) -- (120:14mm) arc (120:150:14mm)--(0,0);
\draw (30:7mm) node[anchor=north west]{$\scriptstyle I_3$};
\draw (150:7mm) node[anchor=north east]{$\scriptstyle I_2$};
\fill[red]  (0,0) circle (0.5mm);
\draw (1cm,-10mm) node{$T_*$};
\end{tikzpicture}
 \\
(a)&\ &(b)
\end{tabular}
\end{center}
\caption{}
\label{fig0}
\end{figure}

There is, however, a more subtle way in which a gluing may fail to  
embed its parts into the whole space. 
To see this, consider another gluing 
of $I_1$, $I_2$ and $I_3$
 depicted on Fig.~\ref{fig}(a) into the 
space $T_\circ$ pictured on Fig.~\ref{fig}(b).
\begin{figure}[!htp]
\begin{center}
\begin{tabular}{ccccc}
\begin{tikzpicture}[scale=0.8]
%\draw[dashed,very thin,rotate=90] arc (40:320:5mm);
\draw[dashed,very thin,red] (0,0) circle (5mm);
\draw[color=white,very thick] (50:5mm) arc (50:130:5mm);
\draw[dashed,very thin,red] (50:20mm) arc (50:130:20mm);
\draw (0,-5mm) -- (0,-20mm);
\draw (50:5mm) -- (50:20mm);
\draw (130:5mm) -- (130:20mm);
\fill[red] (0,-5mm) circle (0.5mm);
\fill[red] (50:5mm) circle (0.5mm);
\fill[red] (130:5mm) circle (0.5mm);
\fill[red] (50:20mm) circle (0.5mm);
\fill[red] (130:20mm) circle (0.5mm);
\draw (0,-5mm) node[anchor=north west]{$\scriptstyle1$};
\draw (0,-20mm) node[anchor=west]{$\scriptstyle-1$};
\draw (0,-13mm) node[anchor=east]{$\scriptstyle I_1$};
\draw (130:5mm) node[anchor=east]{$\scriptstyle1$};
\draw (130:20mm) node[anchor=east]{$\scriptstyle-1$};
\draw (130:13mm) node[anchor=north east]{$\scriptstyle I_2$};
\draw (50:5mm) node[anchor=west]{$\scriptstyle1$};
\draw (50:20mm) node[anchor=west]{$\scriptstyle-1$};
\draw (50:13mm) node[anchor=north west]{$\scriptstyle I_3$};
\end{tikzpicture}&& 
\begin{tikzpicture}[scale=0.8]
\draw (0,0) circle (8mm);
\draw (0,-8mm) -- (0,-18mm);
\fill[red]  (0,-8mm) circle (0.5mm);
\fill[red]  (0,8mm) circle (0.5mm);
\draw (0,-18mm) node[anchor=north]{$\scriptstyle-1$};
\draw (0,-8mm) node[anchor=south]{$\scriptstyle1$};
\draw (0,8mm) node[anchor=south]{$\scriptstyle-1$};
\draw (0,-13mm) node[anchor=west]{$\scriptstyle I_1$};
\draw (-8mm,0) node[anchor=east]{$\scriptstyle I_2$};
\draw (8mm,0) node[anchor=west]{$\scriptstyle I_3$};
\draw (1.1cm,-1.1cm) node{$T_\circ$};
\end{tikzpicture}
&&
\begin{tikzpicture}[scale=0.8]
\draw[dashed,very thin,red] (0,0) circle (5mm);
\draw[dashed,very thin,red] (50:20mm) arc (50:130:20mm);
\draw (0,-5mm) -- (0,-20mm);
\draw (50:5mm) -- (50:20mm);
\draw (130:5mm) -- (130:20mm);
\fill[red] (0,-5mm) circle (0.5mm);
\fill[red] (50:5mm) circle (0.5mm);
\fill[red] (130:5mm) circle (0.5mm);
\fill[red] (50:20mm) circle (0.5mm);
\fill[red] (130:20mm) circle (0.5mm);
\draw (0,-5mm) node[anchor=north west]{$\scriptstyle1$};
\draw (0,-20mm) node[anchor=west]{$\scriptstyle-1$};
\draw (0,-13mm) node[anchor=east]{$\scriptstyle I_1$};
\draw (130:5mm) node[anchor=east]{$\scriptstyle1$};
\draw (130:20mm) node[anchor=east]{$\scriptstyle-1$};
\draw (130:13mm) node[anchor=north east]{$\scriptstyle I_2$};
\draw (50:5mm) node[anchor=west]{$\scriptstyle1$};
\draw (50:20mm) node[anchor=west]{$\scriptstyle-1$};
\draw (50:13mm) node[anchor=north west]{$\scriptstyle I_3$};
\end{tikzpicture}
 \\
(a)&\ &(b) &\ & (c)
\end{tabular}
\end{center}
\caption{}
\label{fig}
\end{figure}
All the $I_j$'s are embedded into $T_\circ$ but the partial gluing of $I_2$ and $I_3$ is not. Of course one can define
an alternative  gluing
procedure of $I_j$'s into $T_\circ$ (see Fig.~\ref{fig}(c)) for which all partial gluings are embedded into $T_\circ$.

Let us now consider the problem of gluing from the point of view of 
algebras.
Let $J$ be a finite set, and let
\begin{equation}\label{family}
\{\pi^i_j:B_i\longrightarrow B_{ij}=B_{ji}\}_{i,j\in J,\,i\neq j}
\end{equation} 
be a family of algebra homomorphisms.
\begin{dfn}[\cite{CM2000,GK-P99}]
The \emph{multi-pullback 
algebra} $B^\pi$ of a family~\eqref{family} of algebra homomorphisms
is defined as
\[
B^\pi:=\left\{\left.(b_i)_i\in\prod_{i\in J}B_i\;\right|\;\pi^i_j(b_i)=\pi^j_i(b_j),\;\forall\, i,j\in J,\, i\neq j
\right\}.
\]
\end{dfn}
\begin{dfn}
A family \eqref{family} of algebra homomorphisms is called 
\emph{distributive} if and only if all of them are surjective and their
kernels generate a distributive lattice of ideals.
\end{dfn}
\noindent
The multi-pullback algebra of a distributive family of homomorphisms
is the main mathematical concept of this note, and plays
a key role in~\cite{hkmz11,hkz12,r-j}. In particular, it includes 
the multi-pullbacks of all finite families of $C^*$-epimorphisms
between unital $C^*$-algebras.
In case of commutative unital $C^*$-algebras, 
such a multi-pullback $C^*$-algebra
can be identified with
 the algebra of all continuous functions on the compact Hausdorff space 
obtained by the gluing procedure
described in the abstract applied to compact Hausdorff spaces.

\begin{exm}\label{ex1}
Consider the $C^*$-algebra $C(T_*)$ of all continuous functions on $T_*$ 
as the multi-pullback $C^*$-algebra corresponding to the gluing depicted 
on Fig.~\ref{fig0}(a).
Here we take $B_i=C(I_i)$, $i=1,2,3$, $B_{12}=B_{13}=\mC$, $B_{23}=\mC\oplus\mC$, and define $C^*$-epimorphisms by the formulae
\begin{equation*}
\pi^1_2=\pi^2_1=\pi^1_3=\pi^3_1:f\mapsto f(1), \quad\pi^2_3:f\mapsto(f(-1),f(1)),\quad \pi^3_2:f\mapsto(f(1),f(-1)).
\end{equation*}
The fact that $I_2$ is not embedded in $T_*$ corresponds to the non-surjectivity
of the canonical projection $B^\pi\rightarrow B_2$.
\end{exm}

\begin{exm}\label{ex2}
Consider the $C^*$-algebra $C(T_\circ)$ 
of all continuous functions on $T_\circ$ 
as the multi-pullback $C^*$-algebra corresponding to the gluing depicted 
on Fig.~\ref{fig}(a).
Here we take $B_i=C(I_i)$, $B_{ij}:=\mC$, $1\leq i,j\leq 3$, $i\neq j$,
and define $C^*$-epimorphisms by the formulae
\begin{equation*}
\pi^1_2=\pi^2_1=\pi^1_3=\pi^3_1:f\mapsto f(1),\quad \pi^2_3=\pi^3_2:f\mapsto f(-1).
\end{equation*}
 While  the canonical projections $B^\pi\rightarrow B_i$ are all
 surjective, the canonical projection $B^\pi\rightarrow\{(b_2,b_3)\in B_2\times B_3\;|\;\pi^2_3(b_2)=\pi^3_2(b_3)\}$ is not.
Indeed, a pair $b_2:=(t\mapsto t)$, $b_3:=(t\mapsto -1)$ satisfies
$\pi^2_3(b_2)=\pi^3_2(b_3)$, but there is no function $b_1\in B_1$ such that $(b_1,b_2,b_3)\in B^\pi$.
This corresponds to the fact that the gluing of $I_2$ and $I_3$ is not embedded in $T_\circ$.
\end{exm}

\begin{exm}\label{ex3}
We can present the $C^*$-algebra $C(T_\circ)$ of all continuous functions
on $T_\circ$ pictured in Fig.~\ref{fig}(b) by using different multi-pullbacks: one corresponding to the gluing depicted
in Fig.~\ref{fig}(a) (see Example~\ref{ex2}) and one corresponding to the gluing depicted
in Fig.~\ref{fig}(c). For the latter case,
 we take the $B_i$'s, $B_{12}$, $B_{13}$,
$\pi^1_2$, $\pi^2_1$, $\pi^1_3$, $\pi^3_1$ as in Example~\ref{ex2}, 
but we put $B_{23}:=\mC\oplus\mC$
and $\pi^2_3=\pi^3_2:f\mapsto (f(-1),f(1))$. Now
not only the  canonical projections 
$B^\pi\rightarrow B_i$ are all surjective, but also,
for all distinct $i,j,k$ and 
all $b_i\in B_i$, $b_j\in B_j$ such that 
$\pi^i_j(b_i)=\pi^j_i(b_j)$, there exists $b_k\in B_k$ such that 
$\pi^i_k(b_i)=\pi^k_i(b_k)$ and $\pi^j_k(b_j)=\pi^k_j(b_k)$.
\end{exm}

It turns out that the cocycle condition defined below
is a perfect tool to understand
the differences between the above examples.
To define the cocycle condition, for any distinct $i,j,k$ we put
$B^i_{jk}:=B_i/(\ker\pi^i_j+\ker\pi^i_k)$ and take 
$[\cdot]^i_{jk}:B_i\rightarrow B^i_{jk}$ to be the canonical surjections.
Next, we introduce the family of isomorphisms
\begin{equation}
\pi^{ij}_k:B^i_{jk}\longrightarrow B_i/\pi^i_j(\ker\pi^i_k),\qquad
[b_i]^i_{jk}\longmapsto\pi^i_j(b_i)+\pi^i_j(\ker\pi^i_k).
\end{equation}
Now we are ready for:
\begin{dfn}
\label{cocycle}
\cite[in Proposition~9]{CM2000}
 We say that a  family~\eqref{family} of algebra
epimorphisms
 %$\{\pi^i_j:B_i\rightarrow B_{ij}\}_{i,j\in J}$ 
satisfies the {\em cocycle condition}
 if and only if, for all distinct $i,j,k\in J$, 
\begin{enumerate}
\item $\pi^i_j(\ker\pi^i_k)=\pi^j_i(\ker\pi^j_k)$,
\item the isomorphisms $\phi^{ij}_k:=(\pi^{ij}_k)^{-1}\circ\pi^{ji}_k:B^j_{ik}\rightarrow B^i_{jk}$  satisfy 
$\phi^{ik}_j=\phi^{ij}_k\circ\phi^{jk}_i$.
\end{enumerate}
\end{dfn}

It was proven in \cite{CM2000} that if a distributive family
of $\pi^i_j$'s
satisfies the cocycle condition, then the canonical projections
$B^\pi\rightarrow B_i$  are all surjective. 
In particular, the multi-pullback from Example~\ref{ex1} cannot satisfy the cocycle condition.
The multi-pullback presentation of $C(T_\circ)$
from Example~\ref{ex2} demonstrates, however, that the cocycle condition is not necessary for the canonical projections
$B^\pi\rightarrow B_i$  to be  surjective. Indeed,
they are all
clearly surjective in this case, but 
$\pi^i_j$'s do not satisfy the cocycle condition
because $\pi^1_2(\ker\pi^1_3)=\{0\}$ whereas
$\pi^2_1(\ker\pi^2_3)=\mC$.
On the other hand, the cocycle condition is satisfied by an alternative multi-pullback presentation  of $C(T_\circ)$
given in Example~\ref{ex3}.
This suggests that the cocycle condition is related to the possibility of extending partial
multi-pullbacks. Thus we arrive at the main result of this note:

\begin{thm}
\label{gen}
The following statements about a distributive family~\eqref{family} of  
algebra homomorphisms are equivalent:
\begin{enumerate}
\item The family~\eqref{family} satisfies the
cocycle condition.
\item For any $K\subsetneq J$,  $k\in J\setminus K$
 and  $(b_l)_{l\in K}\in\prod_{l\in K}B_l$
such that $\pi^i_j(b_i)=\pi^j_i(b_j)$ for all distinct $i,j\in K$, 
there exists $b_k\in B_k$ such that
$\pi^l_k(b_l)=\pi^k_l(b_k)$ for all $l\in K$.
\item For all distinct $i,j,k\in J$ and all
 $b_i\in B_i$, $b_j\in B_j$ such that 
$\pi^i_j(b_i)=\pi^j_i(b_j)$, there exists $b_k\in B_k$ such that also 
$\pi^i_k(b_i)=\pi^k_i(b_k)$ and $\pi^j_k(b_j)=\pi^k_j(b_k)$.
\end{enumerate}
\end{thm}
\begin{proof}
The proof of $(1)\Rightarrow (2)$ is essentially identical with the proof
of \cite[Proposition~9]{CM2000}, and $(3)$ is obviously a 
special case  of $(2)$.
In order to prove $(3)\Rightarrow (1)$ and close the loop of implications,
assume that for any distinct $i,j,k\in J$ and for arbitrary elements $b_i\in B_i$ and $b_j\in B_j$ such that 
$\pi^i_j(b_i)=\pi^j_i(b_j)$ there exists $b_k\in B_k$ such that also $\pi^i_k(b_i)=\pi^k_i(b_k)$ and $\pi^j_k(b_j)=\pi^k_j(b_k)$.
Specializing this condition for $b_j=0$ yields that for any $b_i\in\ker\pi^i_j$ there exists a
$b_k\in\ker\pi^k_j$ such that $\pi^i_k(b_i)=\pi^k_i(b_k)$, that is
 $\pi^i_k(\ker\pi^i_j)\subseteq\pi^k_i(\ker\pi^k_j)$.
Exchanging $i$ and $k$ we obtain the set equality. This proves
Condition~(1) defining the cocycle condition.

To prove the second condition observe  that, for all
distinct $i,j,k\in J$ and any $b_i\in B_i$, $b_j\in B_j$,
\begin{equation}
\label{eqphi}
[b_i]^i_{jk}=\phi^{ij}_k([b_j]^j_{ik})\;\Leftrightarrow\;
\pi^{ji}_k([b_j]^j_{ik})=\pi^{ij}_k([b_i]^{i}_{jk})
\;\Leftrightarrow\;
\pi^i_j(b_i)-\pi^j_i(b_j)\in\pi^i_j(\ker\pi^i_k).
\end{equation}
Now let us pick any distinct $i,j,k\in J$   and any $b_j\in B_j$.
Since $\pi^k_j$ is surjective, there exists $b_k\in B_k$ such that
$\pi^k_j(b_k)=\pi^j_k(b_j)$, so that 
$[b_k]^k_{ji}=\phi^{kj}_i([b_j]^j_{ik})$ by~\eqref{eqphi}. 
Furthermore, by assumption, there exists
$b_i\in B_i$ such that $\pi^i_k(b_i)=\pi^k_i(b_k)$ and $\pi^i_j(b_i)=\pi^j_i(b_j)$. Therefore, again by~\eqref{eqphi}, we obtain
\begin{equation}
[b_i]^i_{jk}=\phi^{ik}_j([b_k]^k_{ji})=\phi^{ik}_j(\phi^{kj}_i([b_j]^j_{ik})) \quad\text{and}\quad
[b_i]^i_{jk}=\phi^{ij}_k([b_j]^j_{ik}).
\end{equation}
Plugging in the second equality to the first one, we get
$\phi^{ij}_k([b_j]^j_{ik})=\phi^{ik}_j(\phi^{kj}_i([b_j]^j_{ik}))$
for any $[b_j]^j_{ik}\in B^j_{ik}$, as needed.
\end{proof}

Finally, let us remark that the fact that in Example~\ref{ex3} we
could remedy the lack of the cocycle condition in Example~\ref{ex2}
is not a coincidence. Indeed, following~\cite[Proposition~4]{CM2000}, 
one sees that, if $B^\pi$ is the 
 multi-pullback  a family~\eqref{family}
such that the canonical projections 
$B^\pi\rightarrow B_i$ are all surjective
 and their kernels generate a distributive lattice of ideals, then 
$B^\pi$ can  also be presented
as the multi-pullback of a family satisfying the cocycle condition 
even if the original family failed to do so. Herein the new family is
defined via  the  canonical surjections:
\begin{equation}
 \{\pi^{'i}_j:B_i\cong B^\pi/\ker(B^\pi\rightarrow B_i) 
\longrightarrow 
B^\pi/
(\ker(B^\pi\rightarrow B_i)+
\ker(B^\pi\rightarrow B_j))\}_{i,j\in J,\,i\neq j}.
\end{equation}
The aforementioned example is a special case of this general claim because
$C^*$-ideals always generate a distributive lattice.

\noindent\textbf{Acknowledgements:}
This work is part of the project \emph{Geometry and Symmetry of Quantum
Spaces} sponsored by the grants 
PIRSES-GA-2008-230836 and 
1261/7.PR~UE/2009/7.

\end{document}